\documentclass{article}
\usepackage[utf8]{inputenc}
\usepackage{amssymb, amsmath, amsthm, mathrsfs}
\usepackage{showkeys}
\usepackage{graphicx}

\newtheorem{theorem}{Theorem}
\newtheorem{lemma}{Lemma}

\newtheorem{prop}{Proposition}

\title{The Landau-Kolmogorov Problem on a Finite Interval in the Taikov Case}
\author{dmitriy.skorokhodov }
\date{November 2020}

\begin{document}

\maketitle

\begin{abstract}
    We solve the pointwise Landau-Kolmogorov problem on the interval $\mathbb{I} = [-1,1]$ on finding $\left|f^{(k)}(t)\right|\to\sup$ under constraints $\|f\|_2 \leqslant \delta$ and $\left\|f^{(r)}\right\|_2\leqslant 1$, where $t\in\mathbb{I}$ and $\delta > 0$ are fixed. For $r = 1$ and $r = 2$, we solve the uniform version of the Landau-Kolmogorov problem on the interval $\mathbb{I}$ in the Taikov case by proving the Karlin-type conjecture $\sup\limits_{t\in \mathbb{I}}\left|f^{(k)}(t)\right| = \left|f^{(k)}(-1)\right|$ under above constraints. The proof relies on the analysis of the dependence of the norm of the solution to higher-order Sturm-Liouville equation $(-1)^ru^{(2r)} + \lambda u = -\lambda f$ with boundary conditions $u^{(s)}(-1) = u^{(s)}(1) = 0$, $s = 0,1,\ldots,r-1$, on non-negative parameter $\lambda$, where $f$ is some piece-wise polynomial function. Furthermore, we find sharp inequality $\left\|f^{(k)}\right\|_\infty \leqslant A\|f\|_2 + B\left\|f^{(r)}\right\|_2$ with the smallest possible constant $A > 0$ and the smallest possible constant $B = B(A)$ for $k \in \{r-2, r-1\}$. ~\\
    MSC2010 (primary): 41A17, 41A35, 26D10.\\
    MSC2010 (secondary): 34B24.
\end{abstract}

\section{Introduction}

The Landau-Kolmogorov problem consists of finding sharp upper bound for the norm of intermediate function derivative in terms of the norm of the function itself and the norm of its higher order derivative. 

Let $\mathbb{I} = [-1,1]$. For $1\leqslant p\leqslant\infty$, let $L_p = L_p(\mathbb{I})$ be the space of measurable functions $f:\mathbb{I}\to\mathbb{R}$ having integrable $p$-th power (essentially bounded when $p=\infty$) with the standard norm:
\[
    \|f\|_p := \|f\|_{L_p(\mathbb{I})} = \left\{\begin{array}{ll}
        \displaystyle \left(\int_{-1}^1|f(x)|^p\,{\rm d}x\right)^{\frac 1p},& 1\leqslant p < \infty, \\ 
        \textrm{esssup}\,\left\{|f(x)|\,:\,x\in\mathbb{I}\right\},& p=\infty
    \end{array}\right.
\]
Let $L_p^r$, $r\in\mathbb{N}$, be the space of functions $f:\mathbb{I}\to \mathbb{R}$ having absolutely continuous derivative $f^{(r-1)}$ and such that $f^{(r)}\in L_p$. By $W^r_p = \left\{f\in L_p^r\,:\,\|f\|_p\leqslant 1\right\}$ we denote the unit ball in space $L_p^r$.

Let $k,r\in\mathbb{Z}_+$ be such that $0\leqslant k \leqslant r-1$ and $1\leqslant p,q,s\leqslant \infty$. {\it The Landau-Kolmogorov problem} consists of finding the modulus of continuity of operator $D^k:L_p\to L_q$ on the class $W^r_s$:
\begin{equation}
\label{modulus}
    \Omega(\delta) = \Omega\left(\delta;D^k;W^r_s\right) = \sup{\left\{\left\|f^{(k)}\right\|_q\,:\,f\in W^r_s,\,\|f\|_p\leqslant\delta\right\}},\quad \delta\geqslant 0.
\end{equation}
%
In the case $q = \infty$ {\it the pointwise Landau-Kolmogorov problem} is also considered. It consists of finding the modulus of continuity of functional $D^k_t:L_p\to\mathbb{R}$, where $t\in\mathbb{I}$, on the class $W^r_s$: 
\begin{equation}
\label{modulus_pointwise}
    \Omega_t(\delta) = \Omega\left(\delta;D^k_t;W^r_s\right) = \sup{\left\{\left|f^{(k)}(t)\right|\,:\,f\in W^r_s,\,\|f\|_p\leqslant\delta\right\}},\quad \delta\geqslant 0.
\end{equation}


Problems~\eqref{modulus} and~\eqref{modulus_pointwise} are the most studied in the case of uniform norms $p=q=s=\infty$. S.~Karlin~\cite{Kar_76} and A.~Pinkus~\cite{Pin_78} found $\Omega_t$ 
for all $t\in \mathbb{I}$ and $k,r\in\mathbb{N}$ with $k\leqslant r-1$. In~\cite{Kar_76} S.~Karlin conjectured that, for every $\delta > 0$,  
\begin{equation}
\label{karlin_conjecture}
    \Omega(\delta) = \sup\limits_{t\in\mathbb{I}}\Omega_t(\delta) = \Omega_{-1}(\delta).
\end{equation}
Although not fully proved, conjecture~\eqref{karlin_conjecture} is confirmed in many situations (see {\it e.g.},~\cite{Sha_14}): for $r=2$ by E.~Landau~\cite{Lan_13} ($\delta < \delta_2$) and C.\,K.~Chui, P.\,W.~Smith~\cite{ChuSmi_75} ($\delta \geqslant \delta_2$); for $r=3$ by M.~Sato~\cite{Sat_82} and A\,I.~Zvyagintsev, A.\,Ya.~Lepin~\cite{ZvyLep_82} independently; for $r=4$ by A.\,I.~Zvyagintsev~\cite{Zvy_82} ($\delta \geqslant \delta_4$) and N.~Naidenov~\cite{Nai_03} ($\delta < \delta_4$). For $r\geqslant 5$, Karlin's conjecture is proved in the {\it ``polynomial case''} by B.-O.~Eriksson~\cite{Eri_98} ($\delta = \delta_r$) and A.\,Yu.~Shadrin~\cite{Sha_14} ($\delta > \delta_r$). Here $\delta_r = \frac{1}{2^{r-1}r!}$.

Much less is known for other combinations of parameters $p,q,s$. Moduli of continuity $\Omega$ and $\Omega_t$ were found for all $\delta > 0$ in the following situations:
\begin{enumerate}
    \item $r\in\mathbb{N}\setminus\{1\}$, $k = r-1$, $q=\infty$, $1\leqslant p\leqslant\infty$, $s=1$ by V.\,I.~Burenkov~\cite{Bur_86};
    \item $r=2$, $k=1$, $t\in\mathbb{I}$, $p=\infty$, $1\leqslant s < \infty$ -- by Yu.\,V.~Babenko~\cite{Bab_00}, V.\,I.~Burenkov and V.\,A.~Gusakov~\cite{BurGus_03} independently;
    \item $r=2$, $k=1$, $p=q=\infty$, $1\leqslant s < \infty$ -- by Yu.\,V.~Babenko~\cite{Bab_00}, V.\,I.~Burenkov and V.\,A.~Gusakov~\cite{BurGus_03} independently;
    \item $r=2$, $k=1$, $p=s=\infty$, $1\leqslant q < \infty$ by N.~Naidenov~\cite{Nai_03_2};
    \item $r=3$, $k=1$, $p=q=\infty$, $r=1$ by the author~\cite{Sko_19}.
\end{enumerate}
Problem~\eqref{modulus} was solved partially in~\cite{BojNai_02} in the case $1\leqslant q < \infty$, $p=s=\infty$, $r=3$ and $k=1,2$. For the overview of other results in this and closely related directions we refer the reader to books~\cite{MitPecFin_91,KwoZet_92,BabKorKofPic_03} and surveys~\cite{Sha_93, Are_96}. 

In this paper we will study the Landau-Kolmogorov problem in the Taikov case, {\it i.e.} for $p=s=2$ and $q=\infty$. The consideration of the Taikov case is motivated by numerous sharp results on the Landau-Kolmogorov problem for functions defined on the real line $\mathbb{R}$ (see L.\,V.~Taikov~\cite{Tai_68}), on the non-negative half-line $\mathbb{R}_+ = [0,+\infty)$ (see V.\,N.~Gabushin~\cite{Gab_69}, G.\,A.~Kalyabin~\cite{Kal_04}, A.\,A.~Lunev and L.\,L.~Oridoroga~\cite{LunOri_09}), the period $\mathbb{T} = [0,2\pi)$ (see A.\,Yu.~Shadrin~\cite{Sha_90}) and other domains. 
These results 
were generalized in various directions: for integral and fractional derivatives of multivariate functions, for powers of the Laplace-Beltrami operators on manifolds, for powers of infinitesimal generators of semigroups, for powers of self-adjoint operators and abstract linear operators acting in Hilbert (see~\cite{BabBabKriSko_21} and references therein). 

Aforementioned advances make it natural to expect that problems~\eqref{modulus} and~\eqref{modulus_pointwise} in the Taikov case can be solved in full. Curiously enough it seems no previous results in this setting are known. 


\subsection{Related problems}

{\sl Additive inequalities for the norms of derivatives.} The Landau-Kolmogorov problem can be formulated alternatively as the problem of finding the set $\Gamma = \Gamma\left(D^k;L_s^r\right)$ of all possible pairs $(A,B)$ of non-negative numbers such that, for every $f\in L_r^s$, there holds inequality
\begin{equation}
\label{inequality}
    \left\|f^{(k)}\right\|_q \leqslant A\|f\|_p + B\left\|f^{(r)}\right\|_s,
\end{equation}
that is sharp in the sense of minimal possible constant $B$, {\it i.e.}
\[
    B = B(A) = \sup\limits_{f\in W^r_s}\left(\left\|f^{(k)}\right\|_q - A\|f\|_p\right).
\]

The pointwise version of problem~\eqref{inequality} naturally consists of finding the set $\Gamma_t = \Gamma\left(D^k_t;L_s^r\right)$ of all pairs $(A,B)$ of non-negative numbers such that, for every $f\in L_r^s$, there holds sharp inequality
\begin{equation}
\label{inequality_pointwise}
    \left|f^{(k)}(t)\right| \leqslant A\|f\|_p + B\left\|f^{(r)}\right\|_s,
\end{equation}
where 
\[
    B = B_t(A) = \sup\limits_{f\in W^r_s}\left(\left|f^{(k)}(t)\right| - A\|f\|_p\right).
\]

Modulus $\Omega$ and set $\Gamma$ are closely related as the following relations indicate
\begin{equation}
\label{first_relation}
    \Gamma = \left\{ (A,B(A))\,:\,A\geqslant 0\,\text{ and }\,B(A) = \sup\limits_{\delta > 0}\left(\Omega\left(\delta\right) - A\delta\right) < \infty \right\}
\end{equation}
\begin{equation}
\label{second_relation}
    \Omega\left(\delta\right) \leqslant \inf\limits_{(A,B)\in \Gamma} \left(A\delta + B\right).
\end{equation}
Inequality~\eqref{second_relation} turns into equality for concave moduli of continuity $\Omega$. For modulus $\Omega_t$ and set $\Gamma_t$ similar relations hold true
\begin{equation}
\label{first_relation_t}
    \Gamma_t = \left\{ (A,B_t(A))\,:\,A\geqslant 0\,\text{ and }\,B_t(A) = \sup\limits_{\delta > 0}\left(\Omega_t\left(\delta\right) - A\delta\right) < \infty \right\}
\end{equation}
\begin{equation}
\label{second_relation_t}
    \Omega_t\left(\delta\right) = \inf\limits_{(A,B)\in \Gamma_t} \left(A\delta + B\right).
\end{equation}
Remark that equality sign in~\eqref{second_relation_t} follows from~\cite[Lemma~1]{Gab_70}. 

Hence, once moduli of continuity $\Omega$ and $\Omega_t$ are known, one can find the sets $\Gamma$ and $\Gamma_t$ immediately. Additional results on sharp additive inequalities~\eqref{inequality} can be found in~\cite{Bur_80}. We also refer the reader to the papers~\cite{Sha_93, BabKofPic_97} and books~\cite{KwoZet_92,BabKorKofPic_03} for the overview of results on inequalities~\eqref{inequality} and~\eqref{inequality_pointwise}.

Remark an important relation between additive Landau-Kolmogorov inequalities and the Markov-Nikolskii inequalities. It was shown independently in~\cite{Sha_98,BabKofPic_97} that the minimal constant $A$ in inequality~\eqref{inequality} coincides with the sharp constant $M = M\left(D^k;\mathcal{P}_{r-1}\right)$ in the Markov-Nikolskii inequality
\[
    \left\|Q^{(k)}\right\|_q \leqslant M\|Q\|_p,\qquad Q\in\mathcal{P}_{r-1},
\]
where $\mathcal{P}_{r-1}$ is the set of all algebraic polynomials of degree at most $r-1$. In other words, $B(A)$ in~\eqref{first_relation} is finite if and only if $A \geqslant M$. Similarly, the minimal constant $A$ in inequality~\eqref{inequality_pointwise} coincides with the sharp constant $M_t = M\left(D^k_t;\mathcal{P}_{r-1}\right)$ in the pointwise version of the Markov-Nikolskii inequality 
\[
    \left|Q^{(k)}(t)\right| \leqslant M_t \|Q\|_p,\qquad Q\in\mathcal{P}_{r-1},
\]
and, equivalently, $B_t(A)$ is finite if and only if $A \geqslant M_t$.



{\sl The best approximation of unbounded operators by linear bounded ones.} We follow S.\,B.~Stechkin~\cite{Ste_65,Ste_67} to formulate this problem. Let $X$ and $Y$ be Banach spaces, $T:X\to Y$ be an operator with domain $\mathcal{D}(T)$ and $W\subset\mathcal{D}(T)$ be some set. Define the modulus of continuity of operator $T$ on the class $W$:
\begin{equation}
\label{modulus_generic}
    \Omega(\delta;T;W) = \sup{\left\{\|Tx\|_Y\,:\,x\in W,\,\|x\|_X\leqslant \delta\right\}},\qquad \delta \geqslant 0.
\end{equation}
Evidently, the notion $\Omega(\delta;T;W)$ generalizes quantities~\eqref{modulus} and~\eqref{modulus_pointwise}.

Let $\mathcal{L}=\mathcal{L}(X,Y)$ be the set of all linear functionals $S:X\to Y$ and define the error of approximation of operator $T$ by operator $S\in\mathcal{L}$ on the class $W$:
\[
    U\left(T;S;W\right) = \sup\limits_{x\in W}\left\|Tx - Sx\right\|_Y.
\]
For $N > 0$, we set
\begin{equation}
\label{stechkin}
    E_N\left(T;W\right) = \inf\limits_{S\in\mathcal{L},\,\|S\|\leqslant N} U\left(T;S;W\right).
\end{equation}
The Stechkin problem on the best approximation of the operator $T$ by linear bounded operators on $W$ consists in finding the quantity~\eqref{stechkin} and extremal operators (if any exists) delivering $\inf$ in the right hand part of~\eqref{stechkin}.

S.\,B.~Stechkin~\cite{Ste_67} (see also~\cite{Are_96,BabKorKofPic_03}) obtained a simple effective lower estimate for~\eqref{stechkin}.

\begin{prop}
\label{thm_steckhin_lower}
    If $T$ is homogeneous (in particular, linear) operator, $W$ is centrally symmetric convex set, then, for $N > 0$ and $\delta\geqslant 0$, 
    \begin{equation}
    \label{stechkin_lower_estimate}
        E_N(T;W) \geqslant \sup\limits_{\delta\geqslant 0} \left(\Omega(\delta;T;W) - N\delta\right).
    \end{equation}
\end{prop}

We refer the reader to the survey~\cite{Are_96} for known results on problem~\eqref{stechkin} and discussion of related questions. Remark that in the case $X = L_p$, $Y = L_q$, $W = W^r_s$, $T = D^k$, by inequality~\eqref{stechkin_lower_estimate},
\begin{equation}
\label{stechkin_lower_estimate_corollary}
    E_N\left(D^k;W^r_s\right) \geqslant \sup\limits_{\delta > 0} \left(\Omega(\delta) - N\delta\right) = B(N),\qquad N \geqslant 0,
\end{equation}
and in the case $X = L_p$, $Y = \mathbb{R}$, $W = W^r_s$, $T = D^k_t$, by~\cite[Lemma~1]{Gab_70}, 
\begin{equation}
\label{pointwise_stechkin_equality}    
    E_N\left(D^k_t;W^r_s\right) = \sup\limits_{\delta > 0} \left(\Omega_t(\delta) - N\delta\right) = B_t(N),\qquad N\geqslant 0.
\end{equation}

{\sl The best recover of operators.} Let us follow~\cite{Are_96} to set the problem rigorously. Let $X$ and $Y$ be the Banach spaces, $T:X\to Y$ be an operator with domain $\mathcal{D}(T)$, $W\subset\mathcal{D}(T)$ be some set. By $\mathscr{R}$ we denote either the set $\mathscr{L}$ of all linear operators acting from $X$ to $Y$, or the set of all mappings $\mathscr{O}$ from $X$ to $Y$. For an arbitrary $\delta\geqslant 0$ and $S\in\mathscr{R}$, we set
\[
    U_\delta\left(T;S;W\right) = \sup{\left\{\left\|Tx - Sy\right\|_Y\,:\,x\in W,\,y\in X,\,\|x-y\|_X\leqslant \delta\right\}}.
\]
The problem of optimal recovery of operator $T$ with the help of set of operators $\mathscr{R}$ on elements of the set $W$ with given error $\delta$ consists of finding the quantity
\begin{equation}
\label{best_recovery}
    \mathcal{E}_\delta(\mathscr{R};T;W) = \inf\limits_{S\in\mathscr{R}}U_\delta(T;S;W).
\end{equation}
The detailed survey of known results and further references can be found {\it e.g.}, in~\cite{Are_96}. The following corollary from~\cite[Theorem~2.1]{Are_96} indicates close relations with the Stechkin problem and problem~\eqref{modulus_generic}.

\begin{prop}
\label{thm_best_recovery}
    If $T$ is homogeneous operator (in particular, linear), $W$ is centrally symmetric convex set, then, for every $N\geqslant 0$ and $\delta\geqslant 0$,
    \[
        \Omega(\delta;T;W)\leqslant \mathcal{E}_\delta(\mathscr{O};T;W)\leqslant\mathcal{E}_\delta(\mathscr{L};T;W) \leqslant \inf\limits_{N > 0} \left(E_N(T;W)+N\delta\right).
    \]
\end{prop}

In the case $X = L_p$, $Y = \mathbb{R}$, $W = W^r_s$, $T = D^k_t$ by Proposition~\ref{thm_best_recovery} and~\eqref{pointwise_stechkin_equality},
\begin{equation}
\label{function_best_recovery}
    \mathcal{E}_\delta\left(\mathscr{O};D^k_t;W^r_s\right)=\mathcal{E}_\delta\left(\mathscr{L};D^k_t;W^r_s\right) = \Omega_t(\delta).
\end{equation}

\subsection{Contribution and organization of the paper} 

In this paper we consider the case $p=s=2$ and $q=\infty$. We obtain the following results:
\begin{itemize}
    \item Find $\Omega_t$, $t\in\mathbb{I}$, for all $r\in\mathbb{N}$ and $k\in\mathbb{Z_+}$, $k\leqslant r-1$;
    \item Find $\Omega$ for $r\in\{1,2\}$ and $k\in\{0,r-1\}$;
    \item Solve related problems~\eqref{inequality_pointwise},~\eqref{inequality},~\eqref{stechkin} and~\eqref{best_recovery} in above cases;
    \item Find sharp inequality of the form~\eqref{inequality} with minimal possible constant $A$ for $r\in\mathbb{N}\setminus\{1\}$ and $k\in\{r-2,r-1\}$.
\end{itemize}

The paper is organized as follows. In Section~\ref{pointwise_t_0} we solve pointwise Landau-Kolmogorov problem in the case $t = \pm 1$, and in Section~\ref{pointwise_t_any} for general case $-1 < t < 1$. The Landau-Kolmogorov problem in the case $q=\infty$ is considered in Section~\ref{uniform}.

For definiteness, everywhere below with except for Section~\ref{uniform} we assume that $r\in\mathbb{N}$ and $k\in\mathbb{Z}_+$, $k\leqslant r-1$, are arbitrary. 

\subsection{Main ideas of the proof}


To solve pointwise problems~\eqref{modulus_pointwise} and~\eqref{inequality_pointwise}, we use S.\,B.~Stechkin's idea~\cite{Ste_65,Ste_67} on intermediate approximation of $f^{(k)}(t)$, $t\in\mathbb{I}$, with the help of bounded functional $S:L_2\to\mathbb{R}$. We consider $S$ of the form
\[
    Sf = \int_{-1}^t w^{(r)}(x)f(x)\,{\rm d}x + \int_t^1 w^{(r)}(x)f(x)\,{\rm d}x,\qquad f\in L_2,
\]
where the function $w = w_{\lambda,t}:\mathbb{I}\to\mathbb{R}$ is chosen in a way that equation
\[
    (-1)^rw^{(2r)}(x) + \lambda w(x) = 0,\qquad x\in(-1,t)\cup(t,1),
\]
holds with some $\lambda\geqslant 0$, and
\[
    f^{(k)}(t) = Sf + \int_{-1}^1 w(x) f^{(r)}(x)\,{\rm d}x,\qquad \forall f\in L_2^r.
\]
Above choice of $w$ together with the Schwarz inequality leads to sharp inequality
\[
    \left|f^{(k)}(t)\right| \leqslant \left\|w^{(r)}\right\|_2\|f\|_2 + \|w\|_2\left\|f^{(r)}\right\|_2,
\]
with extremal function $w^{(r)}$ in the case $\lambda > 0$. Following the higher-order Sturm-Liouville theory and Fourier analysis with respect to eigen-functions of operator $\mathcal{A} = (-1)^rD^{2r}$ with boundary conditions $\mathcal{B}: u^{(s)}(-1) = u^{(s)}(1) = 0$, $s=0,1,\ldots,r-1$, we will show that the norm $\left\|w^{(r)}\right\|_2$ attains all values between the sharp constant $M_t$ in the Markov-Nikolskii inequality $\left|Q^{(k)}(t)\right|\leqslant M_t\|Q\|_2$, $Q\in\mathcal{P}_{r-1}$, and $+\infty$. Remark that operator $S$ above provides the solution to the Stechkin problem~\eqref{stechkin} for functional $T = D^k_t:L_2\to\mathbb{R}$ and class $W = W^r_2$.

To solve problems~\eqref{modulus} and~\eqref{inequality} we conjecture (see~\eqref{extremal_eigenfunctions}) that the derivative $\varphi^{(r+k)}$ of every eigen-function of operator $\mathcal{A}$ with boundary conditions $\mathcal{B}$ attains its maximal absolute value at the endpoints $\pm 1$. We prove this conjecture in the case $r\in\{1,2\}$ only. For $r\geqslant 3$,  conjecture~\eqref{extremal_eigenfunctions} looks also plausible as graphs~\ref{X_figure_4} and~\ref{X_figure_6} indicate. Using this conjecture we can prove that for the same $\lambda\geqslant 0$, $\left\|w_{\lambda,t}\right\|_2 \leqslant \left\|w_{\lambda,-1}\right\|_2$ and $\left\|w_{\lambda,t}^{(r)}\right\|_2 \leqslant \left\|w_{\lambda,-1}^{(r)}\right\|_2$, which together with equality~\eqref{second_relation_t} proves that $\Omega(\delta) = \Omega_{-1}(\delta)$.

\section{Case $t=-1$}
\label{pointwise_t_0}

For $\lambda \geqslant 0$, consider boundary value problem 
\begin{equation}
\label{Sturm-Liouville}
    \left\{\begin{array}{ll}
        (-1)^ru^{(2r)}(x) + \lambda u(x) = 0,& x\in(-1,1),\\
        u^{(s)}(-1) = (-1)^{k-1}\delta_{r-k-1,s},& s = 0,1,\ldots,r-1,\\
        u^{(s)}(1) = 0,& s = 0,1,\ldots,r-1,
    \end{array}\right.
\end{equation}
where $\delta_{i,j}$ is the Kronecker symbol. By $u = u_\lambda \in L_2^{2r}$ denote a solution to problem~\eqref{Sturm-Liouville}. Some properties of functions $u_\lambda$ are summarized in the following. 

\begin{lemma}
\label{pointwise_t_0_representation}
    Let $r\in\mathbb{N}$ and $k\in\mathbb{Z}_+$, $k\leqslant r-1$. Then 
    \begin{enumerate}
        \item  problem~\eqref{Sturm-Liouville} has a unique solution $u = u_\lambda\in L_2^{2r}$ for every $\lambda \geqslant 0$;
        \item the function $\left\|u_\lambda^{(r)}\right\|_2$ continuously increases in $\lambda$ and $\left\|u^{(r)}_0\right\|_2 = M_{-1}$;
        \item the function $\left\|u_\lambda\right\|_2$ continuously decreases in $\lambda$ and $\lim\limits_{\lambda\to +\infty} \left\|u_\lambda\right\|_2 = 0$.
    \end{enumerate}
\end{lemma}

Remark that in~\cite{KozMaz_97} questions on solvability of boundary value problems close to problem~\eqref{Sturm-Liouville} and properties of their solutions were studied.

The following result gives the solution to the Landau-Kolmogorov problem~\eqref{modulus_pointwise} for the functional $D_{-1}^k:L_2\to \mathbb{R}$ in the Taikov case and to the problem of the best recovery of $D^k_{-1}$ on class $W^r_2$ whose elements are given with an error.

\begin{theorem}
\label{thm_landau_kolmogorov_t_0}
    Let $r\in\mathbb{N}$ and $k\in\mathbb{Z}_+$, $k\leqslant r-1$, and $\mathscr{R} = \mathscr{O}$ or $\mathscr{R} = \mathscr{L}$. Then, for every $\delta > 0$, there exists unique $\lambda = \lambda(\delta) > 0$ such that $\delta\lambda\cdot \left\|u_\lambda\right\|_2 = \left\|u_\lambda^{(r)}\right\|_2$, and there hold equalities 
    \[
        \Omega_{-1}(\delta) := \Omega\left(\delta; D_{-1}^k;W_2^r\right) = \mathcal{E}_\delta\left(\mathscr{R};D^k_{-1};W^r_2\right) = \left\|u_\lambda^{(r)}\right\|_2\delta + \left\|u_\lambda\right\|_2.
    \]
\end{theorem}

The set of pairs of sharp constants in additive inequalities~\eqref{inequality_pointwise} is described by the following result.

\begin{theorem}
\label{thm_pointwise_t_0}
    Let $r\in\mathbb{N}$ and $k\in\mathbb{Z}_+$, $k\leqslant r-1$. Then 
    \[
        \Gamma_{-1} := \Gamma\left(D_{-1}^k;L_2^r\right) = \left\{\left(\left\|u_\lambda^{(r)}\right\|_2, \left\|u_\lambda\right\|_2\right),\,\lambda\geqslant 0\right\}.
    \]
\end{theorem}

The following result delivers the solution to the problem on the best approximation of functional $D^k_{-1}:L_2\to\mathbb{R}$ by linear bounded ones on class $W^r_2$.

\begin{theorem}
\label{stechkin_pointwise_t_0}
    Let $r\in\mathbb{N}$ and $k\in\mathbb{Z}_+$, $k\leqslant r-1$. For $N\geqslant M_{-1}$, let $\lambda = \lambda_N\geqslant 0$ be such that $N = \left\|u_{\lambda_N}^{(r)}\right\|_2$ and consider functional $S_{N,-1}:L_2\to\mathbb{R}$:
    \[
        S_{N,-1}f = \int_{-1}^1 u^{(r)}_{\lambda_N}(x)f(x)\,{\rm d}x,\qquad f\in L_2.
    \]
    Then, for $N\in\left[0, M_{-1}\right)$, $E_N\left(D^k_{-1};W^r_2\right) = +\infty$, and, for $N\geqslant M_{-1}$, 
    \[
        E_N\left(D^k_{-1};W^r_2\right) = U\left(D^k_{-1};S_{N,-1};W^r_2\right) = \left\|u_{\lambda_N}\right\|_2.
    \]
\end{theorem}

Remark, that $\Omega_1 = \Omega_{-1}$ and $\Gamma_1 = \Gamma_{-1}$ due to symmetry considerations. Hence, results similar to Theorems~\ref{thm_landau_kolmogorov_t_0},~\ref{thm_pointwise_t_0} and~\ref{stechkin_pointwise_t_0} hold true in the case $t = 1$ with the function $u_\lambda$ being replaced with the function $(-1)^{r-k}u_\lambda(-x)$, $x\in\mathbb{I}$.

\subsection{Proof of Lemma~\ref{pointwise_t_0_representation}}
\label{ss_lemma_1}

Consider the space
\[
    \mathcal{L} := \left\{u\in L_2^{2r}\,:\,u^{(s)}(-1) = u^{(s)}(1) = 0,\; s=0,1,\ldots,r-1\right\}
\]
and linear operator $\mathcal{A} : L_2\to L_2$ with domain $\mathcal{L}$ mapping a function $u\in \mathcal{L}$ into the function $\mathcal{A}u = (-1)^ru^{(2r)}$. Operator $\mathcal{A}$ possesses the following properties:
\begin{itemize}
    \item {\it $\mathcal{A}$ is self-adjoint}. Indeed, for $u,v\in\mathcal{L}$, integrating by parts, we obtain
    \[
        (\mathcal{A}u, v) = \int_{-1}^1 (-1)^{r} u^{(2r)}(x) v(x)\,{\rm d}x = \int_{-1}^1u(x)(-1)^rv^{(2r)}(x)\,{\rm d}x = (u, \mathcal{A}v).
    \]
    \item {\it $\mathcal{A}$ is coercive}. Indeed, for every $u\in \mathcal{L}$ and $x\in\mathbb{I}$, expanding function $u$ with the help of the Taylor formula with the remainder in the integral form and applying the Schwarz inequality, we obtain
    \[
        \left|u(x)\right| = \left|\int_{-1}^x \frac{(x-\xi)^{r-1}}{(r-1)!} u^{(r)}(\xi)\,{\rm d}\xi\right| \leqslant 2^{r-1}\left|\int_{-1}^1 u^{(r)}(\xi)\,{\rm d}\xi\right| \leqslant \frac{2^r}{\sqrt{2}}\left\|u^{(r)}\right\|_2.
    \]
    Hence, $\left\|u\right\|_2 \leqslant 2^r\left\|u^{(r)}\right\|_2$ and, integrating by parts, we have
    \[
        (\mathcal{A}u,u) = \int_{-1}^1(-1)^ru^{(2r)}(x)u(x)\,{\rm d}x = \int_{-1}^1\left(u^{(r)}(x)\right)^2\,{\rm d}x \geqslant \frac{\|u\|_2^2}{2^{2r}}. 
    \]
    \item {\it $\mathcal{A}^{-1}$ is Hilbert-Schmidt operator}. Since $\mathcal{A}$ is coercive, its inverse $\mathcal{A}^{-1}$ is well-defined on the range $\mathcal{A}(\mathcal{L})$. Moreover, $\mathcal{A}^{-1}$ can be expressed explicitly as follows: for every $u\in L_2$ and $x\in\mathbb{I}$, 
    \begin{equation}
    \label{integral_representation}
        \mathcal{A}^{-1}u(x) = (-1)^r\int_{-1}^1 K(x,\xi) u(\xi)\,{\rm d}\xi,
    \end{equation}
    with the kernel 
    \[
        K(x,\xi) = \frac{(x-\xi)_+^{2r-1}}{(2r-1)!} - \left(\frac{(1-\xi)^{2r-1}}{(2r-1)!},\frac{(1-\xi)^{2r-2}}{(2r-2)!},\ldots,\frac{(1-\xi)^{r}}{r!}\right) \cdot F(x),
    \]
    where $f_+ = \max\{f;0\}$ and $F(x)$ is the column-vector
    \[
        F(x) = \left(\begin{array}{cccc}
            \frac{2^r}{r!} & \frac {2^{r-1}}{(r-1)!} & \ldots & \frac{2^1}{1!}\\
            \frac{2^{r+1}}{(r+1)!} & \frac{2^r}{r!} & \ldots & \frac{2^2}{2!}\\
            \vdots & \vdots & \ddots & \vdots\\
            \frac{2^{2r-1}}{(2r-1)!} & \frac{2^{2r-2}}{(2r-2)!} & \ldots & \frac{2^r}{r!}
        \end{array}\right)^{-1}\left(\begin{array}{c}\frac{(x+1)^r}{r!}\\\frac{(x+1)^{r+1}}{(r+1)!}\\\vdots\\\frac{(x+1)^{2r-1}}{(2r-1)!}\end{array}\right).
    \]
    Indeed, denote the right hand part of~\eqref{integral_representation} by $f$. Clearly, $f(-1) = f'(-1) = \ldots = f^{(r-1)}(-1) = 0$. For $s=0,1,\ldots,r-1$, $F^{(s)}(1)$ is column-vector having $1$ in the $(s+1)$-th row and $0$'s in all other rows. Hence, $f^{(s)}(1) = 0$, $s = 0,1,\ldots,r-1$. Hence, $f\in\mathcal{L}$ and $\mathcal{A}f = (-1)^r f^{(2r)} = u$. 
    
    Next, since the kernel $K$ is bounded, we have $\int_{-1}^1\int_{-1}^1 K^2(x,\xi)\,{\rm d}x\,{\rm d}\xi < \infty$. Therefore, operator $\mathcal{A}^{-1}:L_2 \to \mathcal{L}$ is Hilbert-Schmidt operator. 
\end{itemize}

Since $\mathcal{A}$ is coercive self-adjoint operator whose range coincide with $L_2$, its inverse $\mathcal{A}^{-1}$ is positive self-adjoint operator. Also, $\mathcal{A}^{-1}$ is Hilbert-Schmidt operator and, hence, compact operator. By~\cite[Section XI \S9, Theorem~1]{Yos_95} $\mathcal{A}^{-1}$ has discrete spectrum $\gamma_1 \geqslant \gamma_2 \geqslant \ldots > 0$ accumulating only at $0$. Moreover, the system $\Phi = \left\{\varphi_n\right\}_{n=1}^\infty\subset C^\infty\cap \mathcal{L}$ of normed eigen-functions, {\it i.e.} $\left\|\varphi_n\right\|_2 = 1$, corresponding to eigen-values $\gamma_n$'s  is basis in $L_2$, and 
\[
    \mathcal{A}^{-1}f = \sum\limits_{n=1}^\infty \gamma_n \left(f, \varphi_n\right) \varphi_n,\qquad f = \sum\limits_{n=1}^\infty \left(f, \varphi_n\right) \varphi_n\in L_2.
\]
Denoting $\lambda_n = \gamma_n^{-1}$, we obtain that, for every function $u\in \mathcal{L}$, 
\[
    \mathcal{A}u = \sum\limits_{n=1}^\infty \lambda_n \left(u,\varphi_n\right)\varphi_n,\qquad u = \sum\limits_{n=1}^\infty \left(u,\varphi_n\right)\varphi_n.
\]

Observe that system $\Phi_r = \left\{\varphi^{(r)}_n\right\}_{n=1}^\infty$ is orthogonal system and $\Phi_r\perp\mathcal{P}_{r-1}$. Indeed, for every $n\in\mathbb{N}$ and $Q \in \mathcal{P}_{r-1}$, 
\[
    \left(Q, \varphi_n^{(r)}\right) = \int_{-1}^1 Q(x) \varphi_n^{(r)}(x)\,{\rm d}x = (-1)^r\int_{-1}^1Q^{(r)}(x)\varphi_n(x)\,{\rm d}x = 0,
\]
and, for every $m,n\in\mathbb{N}$,
\[
    \left(\varphi_n^{(r)},\varphi_m^{(r)}\right) = \int_{-1}^1\varphi_n(x)(-1)^r\varphi_m^{(2r)}(x)\,{\rm d}x = \left(\varphi_n,\varphi_m\right) = \lambda_m\delta_{m,n}.
\]

Now, consider the case $\lambda = 0$. Clearly, there exists a polynomial $u_0\in\mathcal{P}_{2r-1}$ satisfying the boundary conditions of problem~\eqref{Sturm-Liouville}. Then the polynomial $u_0^{(r)}$ is extremal in the Markov-Nikolskii inequality
\[
    \left|Q^{(k)}(-1)\right| \leqslant M_{-1} \|Q\|_2,\qquad Q\in\mathcal{P}_{r-1}.
\]
Indeed, integrating by parts and applying the Schwarz inequality, we obtain
\[
    \left|Q^{(k)}(-1)\right| = \left|\int_{-1}^1Q(x)u_0^{(r)}(x)\,{\rm d}x\right| \leqslant \left\|u_0^{(r)}\right\|_2\cdot \|Q\|_2,\qquad\forall Q\in\mathcal{P}_{r-1}.
\]
Above inequality turns into equality on polynomial $u_0^{(r)}$. Hence, $\left\|u_0^{(r)}\right\|_2 = M_{-1}$.

Next, let $\lambda > 0$. Substituting $u = v + u_0$ into problem~\eqref{Sturm-Liouville}, we obtain that the function $v$ belongs to the space $\mathcal{L}$ and satisfies equation
\begin{equation}
\label{Sturm-Liouville-inhom}
    (-1)^rv^{(2r)} + \lambda v = -\lambda u_0.
\end{equation}
It is not difficult to see that the function
\[
    v = \sum\limits_{n=1}^\infty \frac{-\lambda \left(u_0,\varphi_n\right)}{\lambda + \lambda_n}\, \varphi_n
\]
delivers the desired solution to equation~\eqref{Sturm-Liouville-inhom}. Indeed, $v\in\mathcal{L}$ as
\[
    \left\|\mathcal{A}v\right\|_2^2 = \sum\limits_{n=1}^\infty \frac{\lambda^2\lambda_n^2\left|\left(u_0,\varphi_n\right)\right|^2}{(\lambda + \lambda_n)^2} \leqslant \lambda^2 \|u_0\|_2^2,
\]
and 
\begin{gather*}
    (-1)^rv^{(2r)} + \lambda v = Av + \lambda v = \sum\limits_{n=1}^\infty \frac{-\lambda \lambda_n (u_0,\varphi_n)}{\lambda + \lambda_n} + \sum\limits_{n=1}^\infty \frac{-\lambda^2 (u_0,\varphi_n)}{\lambda + \lambda_n} \\ = -\lambda \sum\limits_{n=1}^\infty \left(u_0,\varphi_n\right)\,\varphi_n = -\lambda u_0.
\end{gather*}

Therefore, the solution $u_\lambda$ to the problem~\eqref{Sturm-Liouville} exists and can be represented in the form of the series
\begin{equation}
\label{series_representation}
    u_\lambda = v + u_0 = \sum\limits_{n=1}^\infty \frac{\lambda_n\left(u_0,\varphi_n\right)}{\lambda + \lambda_n}\,\varphi_n.
\end{equation}
Formula~\eqref{series_representation} holds in the case $\lambda = 0$ as well. Uniqueness of $u_\lambda$ follows from the fact that the difference of any two distinct solutions to problem~\eqref{Sturm-Liouville} belongs to $\mathcal{L}$ and is an eigen-function of operator $\mathcal{A}$ corresponding to some non-positive eigen-value $-\lambda$, which is impossible.

Finally, we turn to the proof of other assertions in Lemma~\ref{pointwise_t_0_representation}. Note that:
\begin{equation}
\label{norm_equality}
    \left\|u_\lambda\right\|_2^2 = \sum\limits_{n=1}^\infty \frac{\lambda_n^2\left|\left(u_0,\varphi_n\right)\right|^2}{\left(\lambda + \lambda_n\right)^2}
\end{equation}
and, since $u_0^{(r)}\in\mathcal{P}_{r-1}$, $\Phi_r$ is orthogonal system and $\Phi_r\perp\mathcal{P}_{r-1}$, we have
\begin{equation}
\label{norm_derivative_equality}
    \left\|u^{(r)}_\lambda\right\|_2^2 = \left\|u_0^{(r)}\right\|_2^2 + \left\|v^{(r)}\right\|_2^2 = \left\|u_0^{(r)}\right\|_2^2 + \sum\limits_{n=1}^\infty \frac{\lambda^2\lambda_n\left|\left(u_0,\varphi_n\right)\right|^2}{(\lambda + \lambda_n)^2}.
\end{equation}

Evidently, $\left\|u_\lambda^{(r)}\right\|_2$ continuously increases in $\lambda$, $\left\|u_0^{(r)}\right\|_2 = M_{-1}$ and $\left\|u_\lambda\right\|_2$ continuously decreases in $\lambda$, and $\lim\limits_{\lambda\to +\infty}\left\|u_\lambda\right\|_2 = 0$. $\qedsymbol$

\subsection{Proofs of main results of Section~\ref{pointwise_t_0}}

\begin{proof}[Proof of Theorem~\ref{thm_pointwise_t_0}]
    By Lemma~\ref{pointwise_t_0_representation}, for every $\lambda\geqslant 0$, there exists a function $u_\lambda\in L_2^{2r}$ delivering the solution to problem~\eqref{Sturm-Liouville}. For every $f\in L_2^r$, 
    \[
        \left|f^{(k)}(-1)\right| \leqslant \displaystyle \left|\int_{-1}^1 u_\lambda^{(r)}(x)f(x)\,{\rm d}x\right| + \left|f^{(k)}(-1) - \int_{-1}^1 u_\lambda^{(r)}(x)f(x)\,{\rm d}x\right| =: I_1 + I_2.
    \]
    Integrating by parts and accounting for boundary conditions of problem~\eqref{Sturm-Liouville}, 
    \begin{gather*}
        \int_{-1}^1u_\lambda^{(r)}(x)f(x)\,{\rm d}x = \sum\limits_{j=0}^{r-1} (-1)^{j}\left.\left(u_\lambda^{(r-1-j)}(x)f^{(j)}(x)\right)\right|_{-1}^1 + (-1)^{r}\int_{-1}^1 u_{\lambda}(x)f^{(r)}(x)\,{\rm d}x \\ 
        = f^{(k)}(-1) + (-1)^{r}\int_{-1}^1u_\lambda(x)f^{(r)}(x)\,{\rm d}x.
    \end{gather*}
    Substituting above relation into $I_2$ and applying the Schwarz inequality, we have
    \begin{equation}
    \label{pointwise_inequality}
        \left|f^{(k)}(-1)\right| \leqslant \left\|u_\lambda^{(r)}\right\|_2\|f\|_2 + \left\|u_\lambda\right\|_2\left\|f^{(r)}\right\|_2.
    \end{equation}
    Let us show that inequality~\eqref{pointwise_inequality} is sharp. For $\lambda > 0$, the function $f_\lambda := u_\lambda^{(r)}$ is extremal in~\eqref{pointwise_inequality} as $\left\|f_\lambda\right\|_2 = \left\|u_\lambda^{(r)}\right\|_2$ and $f_\lambda^{(r)} = u^{(2r)}_\lambda = (-1)^{r-1}\lambda u_\lambda$ and $\left\|f_\lambda^{(r)}\right\|_2 = \lambda \left\|u_\lambda\right\|_2$, and
    \begin{gather*}
        f^{(k)}_\lambda(-1) = \int_{-1}^1 u_\lambda^{(r)}(x)f_\lambda(x)\,{\rm d}x + \left(f^{(k)}_\lambda(-1) - \int_{-1}^1 u_\lambda^{(r)}(x)f_\lambda(x)\,{\rm d}x\right)\\
        = \displaystyle \left\|u_\lambda^{(r)}\right\|_2^2 + (-1)^{r-1}\int_{-1}^1 u_\lambda(x)f_\lambda^{(r)}(x)\,{\rm d}x \\
        = \left\|u_\lambda^{(r)}\right\|_2^2 + \lambda \left\|u_\lambda\right\|_2^2 = \left\|u_\lambda^{(r)}\right\|_2\left\|f_\lambda\right\|_2 + \left\|u_\lambda\right\|_2\left\|f_\lambda^{(r)}\right\|_2.
    \end{gather*}
    In the case $\lambda = 0$, consider the limit and apply relation~\eqref{norm_derivative_equality}:
    \begin{gather*}
        \lim\limits_{\mu\to 0^+} \frac{f_\mu^{(k)}(-1) - \left\|u_0^{(r)}\right\|_2\left\|f_\mu\right\|_2}{\left\|f_\mu^{(r)}\right\|_2} = \lim\limits_{\mu\to 0^+} \frac{\left\|u_\mu^{(r)}\right\|_2^2 + \mu\left\|u_\mu\right\|_2^2 - \left\|u_0^{(r)}\right\|_2\left\|u_\mu^{(r)}\right\|_2}{\mu \left\|u_0\right\|_2}\\
        = \left\|u_0\right\|_2 + \lim\limits_{\mu\to 0+}\frac{\left\|u_\mu^{(r)}\right\|_2\left(\left\|u_\mu^{(r)}\right\|_2^2 - \left\|u_0^{(r)}\right\|_2^2\right)}{4\mu\left\|u_0\right\|_2 \left\|u_0^{(r)}\right\|_2} \\
        = \left\|u_0\right\|_2 + \frac{1}{4\left\|u_0\right\|_2}\lim\limits_{\mu\to 0+}\frac{1}{\mu}\sum\limits_{n=1}^\infty \frac{\mu^2\lambda_n\left|\left(u_0,\varphi_n\right)\right|^2}{(\mu + \lambda_n)^2} = \left\|u_0\right\|_2.
    \end{gather*}
    Therefore, inequality~\eqref{pointwise_inequality} is sharp. It remains to show that $\left\|u_\lambda^{(r)}\right\|_2$ attains all values in $\left[M_{-1}, +\infty\right)$. By Lemma~\ref{pointwise_t_0_representation},  $\left\|u_\lambda^{(r)}\right\|_2$ continuously increases in $\lambda$ and $\left\|u_0^{(r)}\right\|_2 = M_{-1}$. So, we need only to prove that $\lim\limits_{\lambda\to+\infty}\left\|u_\lambda^{(r)}\right\|_2 = +\infty$. For $n\in\mathbb{N}$, consider the function $f_n = \cos{\left(\pi n ((\cdot)+1) +\frac{\pi k}{2}\right)}$. Clearly, $\|f_n\|_2 \leqslant 1$, $\left\|f_n^{(r)}\right\|_2\leqslant (\pi n)^r$ and $\left|f_n(-1)\right| = (\pi n)^k$. Assume there exists $C > 0$ such that $\left\|u_\lambda^{(r)}\right\|_2 < C$, for every $\lambda > 0$. Choose $n\in\mathbb{N}$ and $\lambda > 0$ such that $(\pi n)^k > 2C$ and $(\pi n)^r\left\|u_\lambda\right\|_2\leqslant C$. Then $\left|f^{(k)}_n(-1)\right| > 2C = C + C > \left\|u^{(r)}_\lambda\right\|_2\left\|f_n\right\|_2 + \left\|u_\lambda\right\|_2\left\|f_n^{(r)}\right\|_2$, which contradicts to inequality~\eqref{pointwise_inequality}. The proof of Theorem~\ref{thm_pointwise_t_0} is finished.
\end{proof}

\begin{proof}[The proof of Theorem~\ref{thm_landau_kolmogorov_t_0}]
    Let us verify that the function 
    \[
        f(\lambda) = \frac{\left\|u_\lambda^{(r)}\right\|_2^2}{\left\|u_\lambda^{(2r)}\right\|^2_2} = \frac{\left\|u_\lambda^{(r)}\right\|_2^2}{\lambda^2\left\|u_\lambda\right\|_2^2},\qquad \lambda > 0,
    \]
    decreases in $\lambda$ and attains all positive values. Indeed, taking into account relations~\eqref{norm_equality} and~\eqref{norm_derivative_equality}, we obtain
    \[
        f(\lambda) = \frac{\left\|u_0^{(r)}\right\|_2^2}{\sum\limits_{n=1}^\infty \frac{\lambda^2\lambda_n^2\left|\left(u_0,\varphi_n\right)\right|^2}{(\lambda + \lambda_n)^2}} + \frac{\sum\limits_{n=1}^\infty \frac{\lambda_n\left|\left(u_0,\varphi_n\right)\right|^2}{(\lambda + \lambda_n)^2}}{\sum\limits_{n=1}^\infty \frac{\lambda_n^2\left|\left(u_0,\varphi_n\right)\right|^2}{(\lambda + \lambda_n)^2}} =: f_1(\lambda) + f_2(\lambda).
    \]
    Clearly, $f_1$ decreases in $\lambda$. Let us show that $f_2$ is non-increasing in $\lambda$. Consider derivative $f_2'$:
    \begin{gather*}
        f'_2(\lambda) = -2\cdot\frac{\sum\limits_{n=1}^\infty \frac{\lambda_n\left|\left(u_0,\varphi_n\right)\right|^2}{(\lambda + \lambda_n)^3}\cdot \sum\limits_{n=1}^\infty \frac{\lambda_n^2\left|\left(u_0,\varphi_n\right)\right|^2}{(\lambda + \lambda_n)^2} - \sum\limits_{n=1}^\infty \frac{\lambda_n\left|\left(u_0,\varphi_n\right)\right|^2}{(\lambda + \lambda_n)^2}\cdot \sum\limits_{n=1}^\infty \frac{\lambda_n^2\left|\left(u_0,\varphi_n\right)\right|^2}{(\lambda + \lambda_n)^3}}{\left(\sum\limits_{n=1}^\infty \frac{\lambda_n^2\left|\left(u_0,\varphi_n\right)\right|^2}{(\lambda + \lambda_n)^2}\right)^2} \\
        = -2 \cdot \frac{\sum\limits_{n,m=1}^\infty \left(\frac{\lambda_n\left|\left(u_0,\varphi_n\right)\right|^2}{(\lambda + \lambda_n)^3}\cdot \frac{\lambda_m^2\left|\left(u_0,\varphi_m\right)\right|^2}{(\lambda + \lambda_m)^2} - \frac{\lambda_n\left|\left(u_0,\varphi_n\right)\right|^2}{(\lambda + \lambda_n)^2}\cdot \frac{\lambda_m^2\left|\left(u_0,\varphi_m\right)\right|^2}{(\lambda + \lambda_m)^3}\right)}{\left(\sum\limits_{n=1}^\infty \frac{\lambda_n^2\left|\left(u_0,\varphi_n\right)\right|^2}{(\lambda + \lambda_n)^2}\right)^2}\\
        = -2 \cdot \frac{\sum\limits_{n,m=1}^\infty \frac{\lambda_n\lambda_m\left|\left(u_0,\varphi_n\right)\right|^2\left|\left(u_0,\varphi_m\right)\right|^2}{(\lambda + \lambda_n)^3(\lambda+\lambda_m)^3} \cdot \lambda_m(\lambda_m - \lambda_n)}{\left(\sum\limits_{n=1}^\infty \frac{\lambda_n^2\left|\left(u_0,\varphi_n\right)\right|^2}{(\lambda + \lambda_n)^2}\right)^2} \\
        = -2 \cdot \frac{\sum\limits_{m=n+1}^\infty\sum\limits_{n=1}^\infty \frac{\lambda_n\lambda_m\left|\left(u_0,\varphi_n\right)\right|^2\left|\left(u_0,\varphi_m\right)\right|^2}{(\lambda + \lambda_n)^3(\lambda+\lambda_m)^3} \cdot (\lambda_m - \lambda_n)^2}{\left(\sum\limits_{n=1}^\infty \frac{\lambda_n^2\left|\left(u_0,\varphi_n\right)\right|^2}{(\lambda + \lambda_n)^2}\right)^2} \leqslant 0.
    \end{gather*}
    Hence, $f_2$ is non-increasing and, as result, $f$ is strictly decreasing on $(0,+\infty)$. Continuity of $f$ follows from Lemma~\ref{pointwise_t_0_representation}. Consider limit cases $\lambda\to 0^+$ and $\lambda\to+\infty$. Observe that as $u_0\not\in\mathcal{L}$, it follows that
    \[
        \sum\limits_{n=1}^\infty \lambda_n^2 \left|\left(u_0,\varphi_n\right)\right|^2 = +\infty.
    \]
    Hence,
    \[
        \lim\limits_{\lambda\to 0+} f(\lambda) = \frac{\left\|u_0^{(r)}\right\|^2_2}{\lambda^2 \left\|u_0\right\|^2_2} = +\infty
    \]
    and
    \[
        \lim\limits_{\lambda \to +\infty} f(\lambda) = \lim\limits_{\lambda \to +\infty} \frac{\sum\limits_{n=1}^\infty \frac{\lambda^2\lambda_n\left|\left(u_0,\varphi_n\right)\right|^2}{(\lambda + \lambda_n)^2}}{\sum\limits_{n=1}^\infty \frac{\lambda^2\lambda_n^2\left|\left(u_0,\varphi_n\right)\right|^2}{(\lambda + \lambda_n)^2}} \leqslant \lim\limits_{\lambda\to +\infty} \sqrt{\frac{\left\|u_\lambda\right\|_2^2}{\sum\limits_{n=1}^\infty \frac{\lambda^2\lambda_n^2\left|\left(u_0,\varphi_n\right)\right|^2}{(\lambda + \lambda_n)^2}}} = 0.
    \]
    From the above and continuity, and monotony of $f$ it follows that, for every $\delta > 0$, there exists unique $\lambda > 0$ such that $\delta^2 = f(\lambda)$. We conclude the argument by observing that by inequality~\eqref{pointwise_inequality},
    \[
        \Omega_{-1}(\delta) \leqslant \left\|u_\lambda^{(r)}\right\|_2 \delta + \left\|u_\lambda\right\|_2,
    \]
    and recalling that the function $f = \frac{u_\lambda^{(r)}}{\lambda\left\|u_\lambda\right\|_2}$ is extremal in~\eqref{pointwise_inequality}, $\|f\|_2 = \delta$ and $f\in W^r_2$. It remains to apply equalities~\eqref{function_best_recovery} to finish the proof.
\end{proof}

\begin{proof}[The proof of Theorem~\ref{stechkin_pointwise_t_0}]
    Clearly, $E_N\left(D_{-1}^k;W^r_2\right) = +\infty$, for every $N < M_{-1}$. For every $N\geqslant M_{-1}$, by Theorem~\ref{thm_pointwise_t_0} there exists $\lambda \geqslant 0$ such that $N = \left\|u_\lambda^{(r)}\right\|_2$. By~\eqref{pointwise_stechkin_equality}, we have 
    \[
        E_N\left(D^k_{-1};W^r_2\right) \geqslant \Omega_{-1}\left(\frac{\left\|u_\lambda^{(r)}\right\|_2}{\lambda\left\|u_\lambda\right\|_2}\right) - N \cdot \frac{\left\|u_\lambda^{(r)}\right\|_2}{\lambda\left\|u_\lambda\right\|_2} = \left\|u_\lambda\right\|_2.
    \]
    To finish the proof, we follow the proof of inequality~\eqref{pointwise_inequality} and obtain 
    \[
        U\left(D_{-1}^k; S_{N,-1};W^r_2\right) = \sup\limits_{f\in W^r_2}\left|f^{(k)}(-1) - \int_{-1}^1 u_\lambda^{(r)}(x)f(x)\,{\rm d}x\right| \leqslant \left\|u_\lambda\right\|_2. \qedhere
    \]
\end{proof}

\section{Case $t\in(-1,1)$}
\label{pointwise_t_any}

In this section we will follow the ideas from the previous section. For $\lambda \geqslant 0$, consider boundary value problem
\begin{equation}
\label{Sturm-Liouville_t}
    \left\{\begin{array}{ll}
        (-1)^ru^{(2r)}(x) + \lambda u(x) = 0, & x\in(-1,t)\cup(t,1),\\
        u^{(s)}(-1) = u^{(s)}(1) = 0, & s=0,1,\ldots,r-1,\\
        u^{(s)}(t+0) - u^{(s)}(t-0) = (-1)^{k-1}\delta_{r-k-1,s}, & s=0,1,\ldots,2r-1.
    \end{array}\right.
\end{equation}
By $u_{\lambda,t}\in L_2^{2r}((-1,t))\cap L_2^{2r}((t,1))$ denote a solution to problem~\eqref{Sturm-Liouville_t}. Note that functions $u_{\lambda,t}^{(s)}$, $s=0,1,\ldots,r-k-2,r-k,\ldots,2r-1$, can be extended by continuity on the interval $\mathbb{I}$. For brevity, we keep notation $u^{(s)}_{\lambda,t}$ for this extension. Some properties of $u_{\lambda,t}$ are summarized in the following proposition.

\begin{lemma}
\label{pointwise_t_representation}
    Let $r\in\mathbb{N}$, $k\in\mathbb{Z}_+$, $k\leqslant r-1$, and $t\in (-1,1)$. Then
    \begin{enumerate}
        \item problem~\eqref{Sturm-Liouville_t} has a unique solution $u_{\lambda,t}\in L_2^{2r}((-1,t))\cap L_2^{2r}((t,1))$, for every $\lambda \geqslant 0$; 
        \item the function $\left\|u_{\lambda,t}^{(r)}\right\|_2$ continuously increases in $\lambda$ and $\left\|u_{0,t}^{(r)}\right\|_2 = M_t$;
        \item the function $\left\|u_{\lambda,t}\right\|_2$ continuously decreases in $\lambda$ and $\lim\limits_{\lambda\to+\infty}\left\|u_{0,t}\right\|_2 = 0$.
    \end{enumerate}
\end{lemma}

The next result delivers the solution to problem~\eqref{modulus_pointwise} in the Taikov case and to the problem of the best recovery of $D^k_t$ on class $W^r_2$ whose elements are given with an error. 

\begin{theorem}
\label{thm_landau_kolmogorov_t}
    Let $r\in\mathbb{N}$, $k\in\mathbb{Z}_+$, $k\leqslant r-1$, $t\in(-1,1)$, and $\mathscr{R} = \mathscr{O}$ or $\mathscr{R} = \mathscr{L}$. For every $\delta > 0$ there exists a unique $\lambda = \lambda(\delta)> 0$ such that $\delta\lambda\cdot \left\|u_{\lambda,t}\right\|_2 = \left\|u^{(r)}_{\lambda,t}\right\|_2$, and there hold the series of equalities
    \[
        \Omega_t(\delta) := \Omega\left(\delta;D_t^k;W^r_2\right) =\mathcal{E}_\delta\left(\mathscr{R};D^k_t;W^r_2\right) = \left\|u^{(r)}_{\lambda,t}\right\|_2\delta + \left\|u_{\lambda,t}\right\|_2.
    \]
\end{theorem}

The set of pairs of sharp constants in additive inequalities~\eqref{inequality_pointwise} can be described as follows.

\begin{theorem}
\label{thm_pointwise_t}
    Let $r\in\mathbb{N}$, $k\in\mathbb{Z}_+$, $k\leqslant r-1$ and $t\in(-1,1)$. Then
    \[
        \Gamma_t := \Gamma\left(D_t^k; L_2^r\right) = \left\{\left(\left\|u^{(r)}_{\lambda,t}\right\|_2,\left\|u_{\lambda,t}\right\|_2\right),\,\lambda\geqslant 0\right\}.
    \]
\end{theorem}

The following result delivers the solution to the problem on the best approximation of functional $D^k_t:L_2\to\mathbb{R}$ by linear bounded ones on class $W^r_2$.

\begin{theorem}
\label{stechkin_pointwise_t}
    Let $r\in\mathbb{N}$, $k\in\mathbb{Z}_+$, $k\leqslant r-1$ and $t\in(-1,1)$. For $N \geqslant M_t$, let $\lambda_{N,t} \geqslant 0$ be such that $N = \left\|u^{(r)}_{\lambda_{N,t},t}\right\|_2$ and consider the functional $S_{N,t}:L_2\to\mathbb{R}$:
    \[
        S_{N,t} f := \int_0^1 u_{\lambda_{N,t},t}^{(r)}(x) f(x)\,{\rm d}x,\qquad f\in L_2.
    \]
    Then, for $N\in\left(0,M_t\right)$, $E_N\left(D^k_t;W^r_2\right) = +\infty$ and, for $N\geqslant M_t$,
    \[
        E_N\left(D^k_t;W^r_2\right) = U\left(D^k_t;S_{N,t};W^r_2\right) = \left\|u_{\lambda_{N,t},t}\right\|_2.
    \]
\end{theorem}

\subsection{The proof of Lemma~\ref{pointwise_t_representation}}

First, we consider the case $\lambda = 0$. Clearly, there exists a function $u_{0,t}\in\mathcal{P}_{2r-1}((-1,t))\cap \mathcal{P}_{2r-1}((t,1))$ delivering the solution to problem~\eqref{Sturm-Liouville_t}. Note that $u_{0,t}^{(r)}$ is polynomial of degree at most $r-1$ and is extremal in pointwise version of the Markov-Nikolskii inequality
\[
    \left|Q^{(k)}(t)\right| \leqslant M_t \|Q\|_2,\qquad \forall Q\in\mathcal{P}_{r-1}.
\]
Indeed, for every $Q\in\mathcal{P}_{r-1}$, there holds true inequality
\[
    \left|Q^{(k)}(t)\right| = \left|\int_{-1}^1 Q(x)u_{0,t}^{(r)}(x)\,{\rm d}x\right| \leqslant \left\|u_{0,t}^{(r)}\right\|_2\cdot \|Q\|_2,
\]
which turns into equality on the polynomial $u_{0,t}^{(r)}$. Hence, $\left\|u_{0,t}^{(r)}\right\|_2 = M_t$.

Next, let $\lambda > 0$. Substituting $u = v + u_{0,t}$ into problem~\eqref{Sturm-Liouville_t} , we obtain that $v\in\mathcal{L}$ (see definition in Subsection~\ref{ss_lemma_1}) and there holds equality
\begin{equation}
\label{Sturm-Liouville-inhom_t}
    (-1)^rv^{(2r)} + \lambda v = -\lambda u_{0,t}.
\end{equation}

Following the arguments in Subsection~\ref{ss_lemma_1}, we see that 
\[
    v = \sum\limits_{n=1}^\infty \frac{-\lambda \left(u_{0,t},\varphi_n\right)}{\lambda + \lambda_n}\, \varphi_n
\]
and, hence,
\begin{equation}
\label{series_representation_t}
    u_{\lambda,t} = v + u_{0,t} = \sum\limits_{n=1}^\infty \frac{\lambda_n\left(u_{0,t},\varphi_n\right)}{\lambda + \lambda_n}\,\varphi_n.
\end{equation}
Note that the formula~\eqref{series_representation_t} also holds true in the case $\lambda = 0$. Uniqueness of $u_{\lambda,t}$ can be established with the help of the same considerations as uniqueness of the function $u_\lambda$ (see Subsection~\ref{ss_lemma_1}). Furthermore, 
\[
    \left\|u_{\lambda,t}\right\|_2^2 = \sum\limits_{n=1}^\infty \frac{\lambda_n^2\left|\left(u_{0,t},\varphi_n\right)\right|^2}{\left(\lambda + \lambda_n\right)^2}\quad\text{and}\quad 
    \left\|u^{(r)}_{\lambda,t}\right\|_2^2 = \left\|u_{0,t}^{(r)}\right\|_2^2 + \sum\limits_{n=1}^\infty \frac{\lambda^2\lambda_n\left|\left(u_{0,t},\varphi_n\right)\right|^2}{(\lambda + \lambda_n)^2}.
\]
Evidently, $\left\|u_{\lambda,t}^{(r)}\right\|_2$ continuously increases in $\lambda$, $\left\|u_{0,t}^{(r)}\right\|_2 = M_t$ and $\left\|u_{\lambda,t}\right\|_2$ continuously decreases in $\lambda$, and $\lim\limits_{\lambda \to +\infty}\left\|u_{\lambda,t}\right\|_2 = 0$. $\qedsymbol$
    
\subsection{Proofs of main results of Section~\ref{pointwise_t_any}}

\begin{proof}[The proof of Theorem~\ref{thm_pointwise_t}]
    By Lemma~\ref{pointwise_t_representation}, for every $\lambda\geqslant 0$, there exists the solution $u_{\lambda,t}\in L_2^{2r}((-1,t))\cap L_2^{2r}((t,1))$ to the problem~\eqref{Sturm-Liouville_t}. Then, for every $f\in L_2^r$, 
    \[
        \left|f^{(k)}(t)\right| \leqslant \displaystyle \left|\int_{-1}^1 u_{\lambda,t}^{(r)}(x)f(x)\,{\rm d}x\right| + \left|f^{(k)}(t) - \int_{-1}^1 u_{\lambda,t}^{(r)}(x)f(x)\,{\rm d}x\right| =: I_1 + I_2.
    \]
    Integrating by parts and accounting for boundary conditions in~\eqref{Sturm-Liouville_t}, we obtain
    \begin{gather*}
        \int_{-1}^1u_{\lambda,t}^{(r)}(x)f(x)\,{\rm d}x\\ 
        = \sum\limits_{j=0}^{r-1} (-1)^{j}\left.\left(u_{\lambda,t}^{(r-1-j)}(x)f^{(j)}(x)\right)\right|_{-1}^{t^-} + \sum\limits_{j=0}^{r-1} (-1)^{j}\left.\left(u_{\lambda,t}^{(r-1-j)}(x)f^{(j)}(x)\right)\right|_{t^+}^1 \\ + (-1)^{r}\int_{-1}^1 u_{\lambda,t}(x)f^{(r)}(x)\,{\rm d}x \\ 
        = f^{(k)}(t) + (-1)^{r}\int_{-1}^1u_{\lambda,t}(x)f^{(r)}(x)\,{\rm d}x.
    \end{gather*}
    Substituting above relation into $I_2$ and applying the Schwarz inequality, we have
    \begin{equation}
    \label{pointwise_inequality_t}
        \left|f^{(k)}(t)\right| \leqslant \left\|u_{\lambda,t}^{(r)}\right\|_2\|f\|_2 + \left\|u_{\lambda,t}\right\|_2\left\|f^{(r)}\right\|_2.
    \end{equation}
    Sharpness of inequality~\eqref{pointwise_inequality_t} can be established in a similar way as sharpness of inequality~\eqref{pointwise_inequality} in Theorem~\ref{thm_pointwise_t_0}. 
    Similarly, to prove that $\left\|u_{\lambda,t}^{(r)}\right\|_2$ attains all values in $\left[M_t, +\infty\right)$, we can follow the same ideas as in the proof of Theorem~\ref{thm_pointwise_t_0} with function $f_n = \cos{\left(\pi n ((\cdot)-t) +\frac{\pi k}{2}\right)}$. 
\end{proof}

\begin{proof}[The proof of Theorem~\ref{thm_landau_kolmogorov_t}]
    The proof follows the proof of Theorem~\ref{thm_landau_kolmogorov_t_0} with corresponding change of the Fourier coefficients $\left(u_0,\varphi_n\right)$ by coefficients $\left(u_{0,t},\varphi_n\right)$ and of the function $u_\lambda$ by the function $u_{\lambda,t}$.
\end{proof}

\begin{proof}[The proof of Theorem~\ref{stechkin_pointwise_t}]
    The proof is identical to the proof of Theorem~\ref{stechkin_pointwise_t_0} after corresponding change of the function $u_\lambda$ with the function $u_{\lambda,t}$ and operator $S_{N,-1}$ with operator $S_{N,t}$.
\end{proof}

\section{Uniform case}
\label{uniform}

First, we formulate the Karlin-type conjectures: for the modulus of continuity~\eqref{modulus} of operator $D^k:L_2\to L_\infty$ on the class $W^r_2$:
\begin{equation}
\label{karlin_omega}
    \Omega(\delta) = \Omega\left(\delta;D^k;W^r_2\right) = \sup\limits_{t\in\mathbb{I}}\Omega_t(\delta) = \Omega_{-1}(\delta),\qquad \forall\delta\geqslant 0,
\end{equation}
and for the set $\Gamma$ of pairs of sharp constants $A, B$ in additive inequalities~\eqref{inequality}:
\begin{equation}
\label{karlin_gamma}
    \Gamma = \Gamma\left(D^k;L_2^r\right) = \Gamma_{-1}.
\end{equation}

Conjectures~\eqref{karlin_omega} and~\eqref{karlin_gamma} are confirmed for small $r$'s in the following results.

\begin{theorem}
\label{thm_landau_kolmogorov}
    Let either $r = 1$ and $k = 0$ or $r = 2$ and $k \in \{0,1\}$, and $\mathscr{R} = \mathscr{L}$ or $\mathscr{R} = \mathscr{O}$. Then, for $\delta\geqslant 0$, 
    \[
        \Omega(\delta) := \Omega\left(\delta;D^k;W^r_2\right) = \mathcal{E}_\delta\left(\mathscr{R};D^k;W^r_2\right) = \Omega_{-1}(\delta).
    \]
\end{theorem}

\begin{theorem}
\label{thm_uniform}
    Let either $r = 1$ and $k=0$ or $r = 2$ and $k \in\{0, 1\}$. Then
    \[
        \Gamma = \Gamma\left(D^k;L^r_2\right) = \Gamma_{-1}.
    \]
\end{theorem}

The key ingredient in proving above theorems is the conjecture on extremal properties of derivatives of eigen-functions $\varphi$ of operator $\mathcal{A}$ (see Subsection~\ref{ss_lemma_1}):
\begin{equation}
\label{extremal_eigenfunctions}
    \left\|\varphi^{(r+k)}\right\|_\infty = \left|\varphi^{(r+k)}(-1)\right|.
\end{equation}
Here $\varphi$ is non-zero function satisfying boundary value problem for some $\lambda > 0$
\begin{equation}
\label{sturm_liouville_bvp}
    \left\{\begin{array}{ll}
        (-1)^r\varphi^{(2r)}(x) = \lambda \varphi(x), & x\in(-1,1),\\
        \varphi^{(s)}(-1) = \varphi^{(s)}(1) = 0, & s=0,1,\ldots,r-1.
    \end{array}\right.
\end{equation}

Conjecture~\eqref{extremal_eigenfunctions} implies conjectures~\eqref{karlin_omega} and~\eqref{karlin_gamma}, as the following proposition indicates.

\begin{lemma}
\label{lem_karlin_as_corollary}
    If for some $r\in\mathbb{N}$ and $k\in\mathbb{Z}_+$, $k\leqslant r-1$, every eigen-function $\varphi$ of problem~\eqref{sturm_liouville_bvp} possesses property~\eqref{extremal_eigenfunctions}, then equalities~\eqref{karlin_omega} and~\eqref{karlin_gamma} hold true.
\end{lemma}

Theorems~\ref{thm_landau_kolmogorov} and~\ref{thm_uniform} follow immediately from Lemma~\ref{lem_karlin_as_corollary} and the next proposition.

\begin{lemma}
\label{lem_small_order}
    Let either $r = 1$ and $k = 0$ or $r=2$ and $k\in\{0,1\}$. Then every non-zero solution $\varphi$ to problem~\eqref{sturm_liouville_bvp} satisfies property~\eqref{extremal_eigenfunctions}.
\end{lemma}

\begin{figure}[ht]
\caption{Graphs of fourth order derivatives of the first six eigen-functions of operator $\mathcal{A}u = u^{(8)}$, case $r = 4$ and $k = 0$.}
\label{X_figure_4}
\centering
\includegraphics[width=1\textwidth]{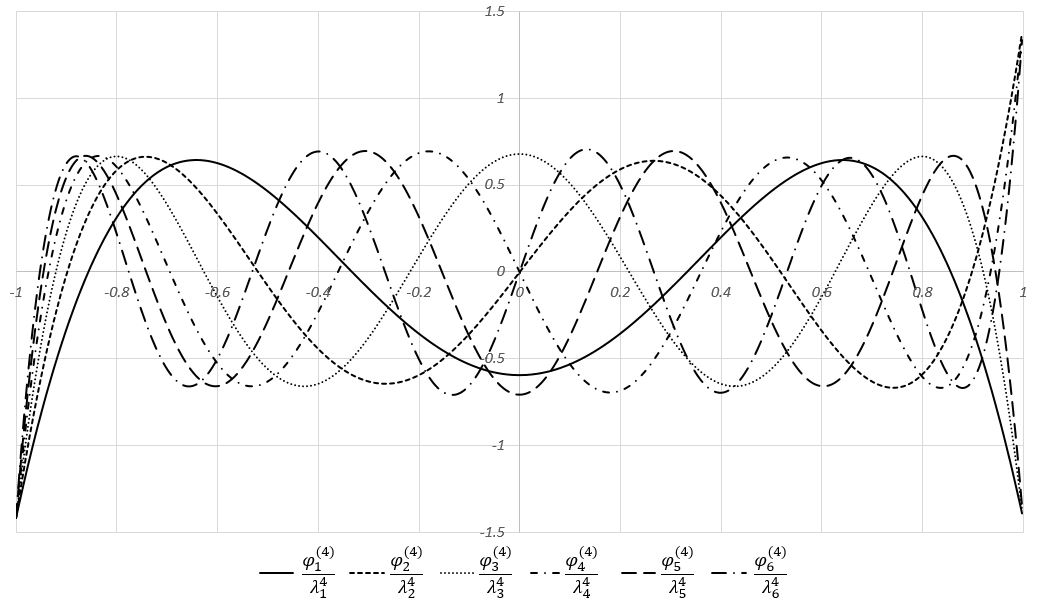}
\end{figure}

We suppose that Lemma~\ref{lem_small_order} also holds true for $r\geqslant 3$ (see {\it e.g.}, graphs of derivatives of first several eigen-functions of operator $\mathcal{A}u = u^{(8)}$), but the proof of this fact is unknown to us.

\begin{figure}[ht]
\caption{Graphs of sixth order derivatives of the first six eigen-functions of operator $\mathcal{A}u = u^{(8)}$, case $r = 4$ and $k = 2$.}
\label{X_figure_6}
\centering
\includegraphics[width=1\textwidth]{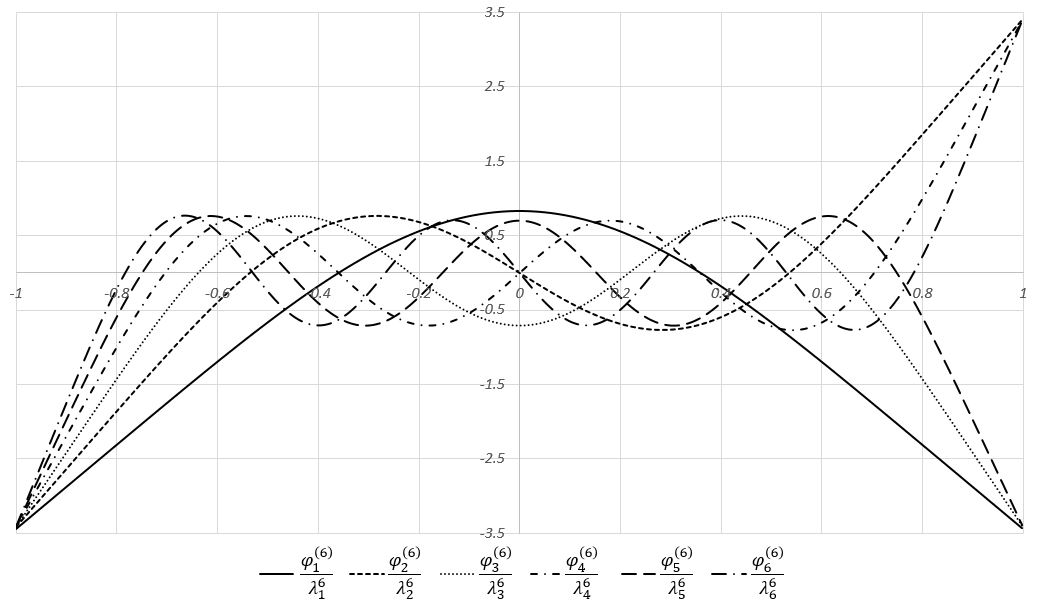}
\end{figure}

Theorem~\ref{thm_uniform} allows solving the Stechkin problem~\eqref{stechkin} for $r = 1$ and $r = 2$.

\begin{theorem}
\label{thm_stechkin}
    Let either $r = 1$ and $k = 0$ or $r = 1$ and $k\in\{0,1\}$. For $N\geqslant M$, let $\lambda_{N,-1} = \lambda_{N,1} = \lambda_N$, where $\lambda_N$ is defined in Theorem~\ref{stechkin_pointwise_t_0}, and let $\lambda_{N,t}$, $t\in(-1,1)$, be defined in Theorem~\ref{stechkin_pointwise_t}. Define operator $S_N:L_2\to L_\infty$:
    \[
        S_Nf(t) = S_{N,t}f = \int_{-1}^1 u_{\lambda_{N,t},t}(x)f(x)\,{\rm d}x,\qquad f\in L_2.
    \]
    Then, for $N \in (0,M)$, $E_N\left(D^k;W^r_2\right) = +\infty$, and, for $N \geqslant M = M_0$, 
    \[
        E_N\left(D^k;W^r_2\right) = U\left(D^k;S_N;W^r_2\right) = \left\|u_{\lambda_N}\right\|_2 = E_N\left(D^k_0;W^r_2\right).
    \]
\end{theorem}

For $r\geqslant 3$, conjectures~\eqref{karlin_omega} and~\eqref{karlin_gamma} remain open. Nevertheless, the following partial result can be established.

\begin{theorem}
\label{thm_additive_polynomial}
    Let $r=3,4,\ldots$ and either $k=r-2$ or $k=r-1$. Then, for every $f\in L^r_2$, there holds sharp inequality
    \[
        \left\|f^{(k)}\right\|_\infty \leqslant \left\|u_{0}^{(r)}\right\|_2\|f\|_2 + \left\|u_0\right\|_2\left\|f^{(r)}\right\|_2.
    \]
\end{theorem}

\subsection{The proof of main results of Section~\ref{uniform}}

\begin{proof}[The proof of Lemma~\ref{lem_karlin_as_corollary}]
    Assume that every eigen-function $\varphi_n$, $n\in\mathbb{N}$, of operator $\mathcal{A}$ satisfies property~\eqref{extremal_eigenfunctions}. First, we show that conjecture~\eqref{karlin_omega} holds true. Integrating by parts and taking into account definition of function $u_{0,t}$ (see Section~\ref{pointwise_t_any}), for $t\in (-1,1)$, we obtain
    \begin{gather*}
        \left(u_{0,t},\varphi_n\right) = \int_{-1}^1 u_{0,t}(x)\varphi_n(x)\,{\rm d}x = \frac{(-1)^r}{\lambda_n} \int_{-1}^1 u_{0,t}(x)\varphi^{(2r)}_n(x)\,{\rm d}x
        = -\frac{\varphi_n^{(r+k)}(t)}{\lambda_n}.
    \end{gather*}
    Similarly, $\varphi_n^{(r+k)}(-1) = -\lambda_n\left(u_0,\varphi_n\right)$, where $u_0$ was defined in Section~\ref{pointwise_t_0}. Hence, $\left|\left(u_{0,t},\varphi_n\right)\right|\leqslant \left|\left(u_{0},\varphi_n\right)\right|$ and by series representation of the norms of functions $u_{\lambda,t}$, $u^{(r)}_{\lambda,t}$, $u_\lambda$ and $u_\lambda^{(r)}$ (see proof of Lemma~\ref{pointwise_t_representation} and Lemma~\ref{pointwise_t_0_representation}), for every $\lambda > 0$, 
    \[
        \left\|u_{\lambda,t}^{(r)}\right\|_2 \leqslant \left\|u_{\lambda}^{(r)}\right\|_2\qquad\text{and}\qquad \left\|u_{\lambda,t}\right\|_2 \leqslant \left\|u_{\lambda}\right\|_2.
    \]
    Then by relation~\eqref{second_relation_t}, for every $\delta \geqslant 0$,
    \[
        \Omega_t(\delta) = \inf\limits_{\lambda \geqslant 0}\left(\left\|u_{\lambda,t}^{(r)}\right\|_2\delta + \left\|u_{\lambda,t}\right\|_2\right) \leqslant \inf\limits_{\lambda\geqslant 0}\left(\left\|u_{\lambda}^{(r)}\right\|_2\delta + \left\|u_{\lambda}\right\|_2\right) = \Omega_{-1}(\delta).
    \]
    As a result, $\Omega(\delta) = \Omega_{-1}(\delta)$ and relation~\eqref{karlin_omega} holds true.
    
    Now, let us show that conjecture~\eqref{karlin_gamma} holds true as well. Let $t\in(-1,1)$. It was proved by G.~Labelle~\cite{Lab_69} that $M = M_0$. Hence, for every $A\geqslant M$, there exist $\lambda \geqslant 0$ such that $\left\|u_{\lambda}^{(r)}\right\|_2 = A$. Then by relations~\eqref{first_relation_t},~\eqref{karlin_omega} and Theorem~\ref{thm_pointwise_t_0}, 
    \[
        B(A) = \sup\limits_{\delta > 0}\left(\Omega(\delta) - A\delta\right) = \sup\limits_{\delta > 0}\left(\Omega_{-1}(\delta) - A\delta\right) = \left\|u_{\lambda}\right\|_2.
    \]
    Therefore, $\Gamma = \Gamma_0$ and conjecture~\eqref{karlin_gamma} is proved.
\end{proof}

\begin{proof}[The proof of Lemma~\ref{lem_small_order}]
    Let $r = 1$ and $k=0$. Then, for $n\in\mathbb{N}$, $\lambda_n = \frac{\pi^2 n^2}{4}$ and, for $m\in\mathbb{N}$, $\varphi_{2m-1}(x) = \cos{\left(\pi m-\frac{\pi}2\right)x}$ and $\varphi_{2m}(x) = \sin{\pi m x}$, $x\in\mathbb{I}$. Clearly, the function $\varphi_n^{(r+k)}(x) = \varphi_n'(x)$ attains its extremal value on $\mathbb{I}$ at the endpoints, which proves~\eqref{extremal_eigenfunctions}.
    
    Let $r = 2$ and $k\in\{0,1\}$. Assume $\varphi$ satisfies boundary value problem~\eqref{sturm_liouville_bvp} with some $\lambda > 0$, and $t_0\in(-1,1)$ is the extremum of $\varphi^{(2+k)}$. Clearly, $\varphi^{(3+k)}(t_0) = 0$. Then integrating by parts we have
    \begin{gather*}
        \left(\varphi^{(2+k)}(t_0)\right)^2 - \left(\varphi^{(2+k)}(-1)\right)^2 = 2 \int_{-1}^{t_0}\varphi^{(2+k)}(x)\varphi^{(3+k)}(x)\,{\rm d}x \\
        = -2 \int_{-1}^{t_0} \varphi^{(1+k)}(x)\varphi^{(4+k)}(x)\,{\rm d}x = -2\lambda \int_{-1}^{t_0}\varphi^{(1+k)}(x)\varphi^{(k)}(x)\,{\rm d}x \\
        = -\lambda \left(\left(\varphi^{(k)}(t_0)\right)^2 - \left(\varphi^{(k)}(-1)\right)^2\right) = -\lambda \left(\varphi^{(k)}(t_0)\right)^2.
    \end{gather*}
    Hence, 
    \[
        \left(\varphi^{(2+k)}(-1)\right)^2 = \left(\varphi^{(2+k)}(t_0)\right)^2 + \lambda \left(\varphi^{(k)}(t_0)\right)^2 \geqslant \left(\varphi^{(2+k)}(t_0)\right)^2,
    \]
    and extremums of $\varphi^{(2+k)}$ inside $\mathbb{I}$ do not exceed in magnitute its values at the end-points of $\mathbb{I}$, which finishes the proof.
\end{proof}



\begin{proof}[The proof of Theorem~\ref{thm_stechkin}]
    Clearly, operator $S_N$ is well-defined for $N \geqslant M$. By Theorems~\ref{stechkin_pointwise_t_0},~\ref{stechkin_pointwise_t} and~\ref{thm_uniform}, we have
    \begin{gather*}
        E_N\left(D^k;W^r_2\right) \leqslant U\left(D^k;S_N;W^r_2\right) = \sup\limits_{t\in\mathbb{I}} U\left(D^k_t;S_{N,t};W^r_2\right) \\ 
        = \sup\limits_{t\in\mathbb{I}} \left\|u_{\lambda_{N,t},t}\right\|_2 = \left\|u_{\lambda_{N}}\right\|_2.
    \end{gather*}
    In turn, by inequality~\eqref{stechkin_lower_estimate} and Theorem~\ref{thm_uniform}, $E_N\left(D^k;W^r_2\right)\geqslant \left\|u_{\lambda_N}\right\|_2$.
\end{proof}

\begin{proof}[The proof of Theorem~\ref{thm_additive_polynomial}]
    Let $t\in (-1,1)$ and the functions $u_\lambda$ and $u_{\lambda,t}$ be defined in Section~\ref{pointwise_t_0} and Section~\ref{pointwise_t_any}, respectively. Since $\left\|u_0^{(r)}\right\|_2 = M_0 \geqslant M_t = \left\|u^{(r)}_{0,t}\right\|_2$ (see~\cite{Lab_69}) and by Theorems~\ref{thm_pointwise_t_0} and~\ref{thm_pointwise_t}, it is sufficient to prove that $\left\|u_{0,t}\right\|_2 \leqslant \left\|u_0\right\|_2$. Also, due to symmetrical considerations, it is enough to consider the case $t\in(-1,0]$.
    
    First, we consider the case $k = r - 1$. Set $p_{-1} := (-1)^r u_0$ and $p_t := (-1)^r u_{0,t}$. By~\eqref{Sturm-Liouville} it is clear that 
    \[
        p_{-1}(\cdot) = 1 - \frac{1}{\gamma}\int_{-1}^{(\cdot)} \left(1-\tau^2\right)^{r-1}\,{\rm d}\tau,\qquad \gamma = \int_{-1}^1\left(1 - \tau^2\right)^{r-1}\,{\rm d}\tau.
    \]
    By~\eqref{Sturm-Liouville_t} $p_t(x) = p_{-1}(x) - \chi_{[-1,t)}$, where $\chi_{E}$ is the characteristic (indicator) function of measurable set $E\subset\mathbb{I}$. Since $t\leqslant 0$, for every $x\in [-1,t)$,
    \[
        p_{-1}(x) \geqslant \frac{1}{2} \geqslant 1 - p_{-1}(x) = -p_t(x) = \left|p_t(x)\right|,
    \]
    and $p_{-1}(x) = p_t(x)$, for every $x\in[t,1]$. Hence, $\left\|p_{-1}\right\|_2 \leqslant \left\|p_t\right\|_2$, which finishes the proof in this case.
    
    Consider the case $r\geqslant 2$ and $k = r-2$. Denote $p_{-1}(x) = (-1)^{r-1}u_{0}(x)$, $p_1(x) = p_{-1}(-x)$ and $p_t(x) = (-1)^{r-1} u_{0,t}(x)$, $x\in\mathbb{I}$. Recall that from the definition of functions $u_0$ and $u_{0,t}$, it follows that $p_{-1}\in\mathcal{P}_{2r-1}$, $p_{t}\in\mathcal{P}_{2r-1}(-1,t)\cap\mathcal{P}_{2r-1}(t,1)$ and
    \[
        p_{-1}^{(s)}(-1) = \delta_{1,s},\quad p_{-1}^{(s)}(1) = 0,\quad s = 0,1,\ldots, r-1,
    \]
    and
    \[
        \left\{\begin{array}{ll}
            p^{(s)}_t(-1) = p^{(s)}_t(1) = 0,&  s=0,1,\ldots,r-1,\\
            p^{(s)}_t\left(t^+\right) - p_t^{(s)}\left(t^{-1}\right) = \delta_{1,s}, & s=0,1,\ldots,2r-1.
        \end{array}\right.
    \]
    Straightforward calculations show that
    \[
        p_t(x) = \frac{(1-t)\cdot p_{-1}(x) + (1+t)\cdot p_1(x) + \delta_t(x)}{2},\qquad x\in\mathbb{I},
    \]
    where
    \[
        \delta_t(x) = \left\{\begin{array}{ll}
            -(1-t)(1+x), & x\in[-1,t],\\
            -(1+t)(1-x), & x\in[t,1].
        \end{array}\right.
    \]
    Clearly, $p_t \leqslant \frac{1-t}{2}\cdot p_{-1} + \frac{1+t}{2}\cdot p_1$ as $\delta_t\leqslant 0$. Let us show that 
    \[
        p_t(x) \geqslant -\frac{1-t}{2}\cdot p_{-1}(x) - \frac{1+t}{2}\cdot p_1(x),\qquad x\in\mathbb{I},
    \]
    or, equivalently,
    \begin{equation}
    \label{main_inequality}
        4p_t(x)\geqslant \delta_t(x),\qquad x\in\mathbb{I}.
    \end{equation}
    
    First, we show that $p_t(t) < 0$. Assume to the contrary that $p_t(t) \geqslant 0$. Then there exist points $-1 < \xi_1 < t < \xi_2 < 1$ such that $p_t'\left(\xi_1\right) \geqslant 0$ and $p_t'\left(\xi_2\right) \leqslant 0$. If $p_t'\left(t^-\right) < 0$ then there exist three points $-1 < \tau_1 < \xi_1 < \tau_2 < t < \xi_2 < \tau_2 < 1$ such that $p_t''\left(\tau_1\right)\geqslant 0$, $p_t''\left(\tau_2\right)\leqslant 0$ and $p_t''\left(\tau_1\right)\geqslant 0$. Hence, the continuous extension of $p_t''$ on the interval $\mathbb{I}$ (also denoted as $p_t''$) has at least $2$ zeros inside $\mathbb{I}$, and at least $2(r-2) + 2 = 2r-2$ zeros (counting multiplicities) on the interval $\mathbb{I}$. However, $p_t''$ is the algebraic polynomial of degree at most $2r-3$ and cannot have $2r-2$ or more zeros. Next, if $p_t'\left(t^-\right) \geqslant 0$ then $p_t'\left(t^+\right) = 1 + p_t'\left(t^-\right) > 0$. Hence, there exist three points $-1 < \tau_1 < \xi_1 < t < \tau_2 < \xi_2 < \tau_2 < 1$ such that $p_t''\left(\tau_1\right)\geqslant 0$, $p_t''\left(\tau_2\right)\leqslant 0$ and $p_t''\left(\tau_1\right)\geqslant 0$, and following previous arguments we arrive to contradiction.
    
    Next, we observe that $p_t(x) \geqslant p_t(t)\cdot\frac{1+x}{1+t}$, $x\in[-1,t]$. Indeed, assume to the contrary that there exists a point $\xi\in(-1,t)$ such that $p_t(\xi) = p_t(t)\cdot \frac{1+\xi}{1+t}$. Since $p_t'(-1) = 0$, there exist $2$ points $-1 < \xi_1 < \xi < \xi_2 < t$ such that $p_t''\left(\xi_1\right) < 0$ and $p_t''\left(\xi_2\right) > 0$. Also, since $p_t(t) < 0$, there exists a point $\xi_3 \in (t,1)$ such that $p_t''(\xi_3) < 0$. Hence, $p_t''$ has at least $2(r-2) + 2 = 2r-2$ zeros (counting multiplicities), which contradicts to the fact that $p_t''$ is a polynomial of degree at most $2r-3$. Using similar arguments we can also prove that $p_t(x)\geqslant p_t(t)\cdot\frac{1-x}{1-t}$, $x\in[t,1]$. 
    
    Based on the above, is sufficient to prove inequality~\eqref{main_inequality} only in the case $x = t$, in which case~\eqref{main_inequality} can be rewritten as
    \begin{equation}
    \label{main_inequality_2}
        2(1-t)\cdot p_{-1}(t) + 2(1+t)\cdot p_1(t) \geqslant -\delta_t(t) = 1 - t^2.
    \end{equation}
    
    Evidently, $p_{-1}(x) \geqslant p_1(x)$ on $\left[-1, 0\right]$ and $p_{-1}(x) \leqslant p_1(x)$ on $\left[0, 1\right]$. Hence, $2(1-t)\cdot p_{-1}(t) + 2(1+t)\cdot p_1(t) \geqslant 2p_{-1}(t) + 2p_1(t)$. So, to prove inequality~\eqref{main_inequality_2} it is sufficient to show that
    \begin{equation}
    \label{main_inequality_3}
        2p_{-1}(t) + 2p_1(t)\geqslant 1- t^2.
    \end{equation}
    Clearly, the function $f(x) = 2p_{-1}(x) + 2p_1(x)$ is even polynomial of degree at most $2r-2$ such that $f(-1) = f(1) = 0$, $f'(-1) = 2$, $f'(1) = -2$ and $f''(x) = -\frac{4}{\gamma} \left(1-x^2\right)^{r-2}$ where $\gamma = \int_{-1}^1\left(1-x^2\right)^{r-2}\,{\rm d}x$. Then inequality~\eqref{main_inequality_3} turns into equality for $r=2$. For $r \geqslant 3$, the difference $g(x) = f(x) - 1 + x^2$ has the following properties: $g(-1) = g(1) = g'(-1) = g'(1) = 0$ and $g''(x)$ is even, has only one zero on $(-1,0)$, and $g''(-1) > 0$. Hence, inequality~\eqref{main_inequality_3} follows, which finishes the proof of Theorem~\ref{thm_additive_polynomial}.
\end{proof}

\end{document}